\documentclass[11pt, reqno]{amsart}

\usepackage{amssymb,latexsym,amsmath,amsfonts,comment}
\usepackage{mathrsfs}
\usepackage{graphicx}
\usepackage[usenames]{color}
\usepackage{hyperref}
\usepackage[normalem]{ulem}

\definecolor{DPurple}{rgb}{0.46,0.2,0.69}

\numberwithin{equation}{section}

\allowdisplaybreaks

\theoremstyle{definition}
\newtheorem{definition}{Definition}[section]

\theoremstyle{remark}
\newtheorem{remark}[definition]{Remark}

\theoremstyle{plain}
\newtheorem{theorem}[definition]{Theorem}
\newtheorem{result}[definition]{Result}

\newtheorem{lemma}[definition]{Lemma}
\newtheorem{proposition}[definition]{Proposition}
\newtheorem{example}[definition]{Example}
\newtheorem{corollary}[definition]{Corollary}

\newtheorem*{fmt}{First Main Theorem}
\newtheorem*{smt}{Second Main Theorem}
\makeatletter

\newcommand{\nf}{{\sf f}}
\newcommand{\nh}{{\sf h}}

\makeatother

\begin{document}

\title[Normal families]{Improvements to the Montel--Carath{\'e}odory Theorem 
 for families of $\mathbb{P}^n$-valued holomorphic curves }

\author{Gopal Datt}
\address{Department of Mathematics, Indian Institute of Science, Bangalore 560012, India}
\email{gopaldatt@iisc.ac.in}



\begin{abstract}
In this paper, we establish various sufficient conditions for a family of 
holomorphic mappings on a domain $D\subseteq\mathbb{C}$ into $\mathbb{P}^n$
to be normal. Our results are improvements to the Montel--Carath{\'e}odory Theorem for a family of 
$\mathbb{P}^n$-valued holomorphic curves. 
\end{abstract}
\maketitle

\vspace{-0.8cm}
\section{Introduction and main results}\label{S:intro}

The work in this paper is influenced, mainly, by the  
Montel--Carath{\'e}odory Theorem and its extensions
in higher dimensions.
The classical Montel--Carath{\'e}odory Theorem states that a family of 
holomorphic $\mathbb{P}^1$-valued mappings on a planar
domain is normal if this family omits three fixed, distinct points in $\mathbb{P}^1$.
This result has gone through various generalizations in one and higher dimensions. 
Xu~\cite{Xu 2010} proved, among the other things, the following result
that improves the classical Montel--Carath{\'e}odory Theorem.

\begin{result}[{\cite[Theorem 1]{Xu 2010}}]\label{R: Xu 2010}
 Let $\mathcal{F}$ be a family of meromorphic functions on a planar domain
 $D \subseteq \mathbb{C}$. Suppose that 
 \begin{enumerate}
  \item[$(a)$] for each pair of functions $f, g\in\mathcal{F}$;
  \begin{itemize}
   \item[$(i)$] $f^{-1}(\{0\})=g^{-1}(\{0\})$, and 
   \item[$(ii)$] $f^{-1}(\{\infty\})=g^{-1}(\{\infty\})$, i.e., $f$ and $g$ have the same set of poles;
  \end{itemize}
  \item [$(b)$]all zeros of $f-1$ are of multiplicity at least $2$ in $D$. 
 \end{enumerate}
 Then $\mathcal{F}$ is normal on $D$.
\end{result}

We asserted that Result~\ref{R: Xu 2010} is an improvement of the classical
Montel--Carath{\'e}odory Theorem because the latter follows as a simple corollary of Result~\ref{R: Xu 2010}.
To see this, suppose a family $\mathcal{F}$ of meromorphic functions omits three distinct values\,---\,say, 
$\alpha, \beta, \gamma$
in $\mathbb{C}\cup\{\infty\}$. Then, consider the following family
\begin{equation*}
\mathcal{G}:=
\Big\{g(z)=\frac{\alpha-\gamma}{\beta-\gamma} \cdot \frac{f(z)-\beta}{f(z)-\alpha} : f\in\mathcal{F}\Big\}
\end{equation*}
(with the understanding that if $\infty\in \{\alpha, \beta, \gamma\}$ then we label it as $\gamma$, and
$(\alpha - \gamma)/(\beta - \gamma)$ is understood to be $1$ in that case).
It is straightforward to see that $g^{-1}\{0\}=g^{-1}\{1\}=g^{-1}\{\infty\}=\emptyset$ for all $g\in\mathcal{G}$. 
Thus, $\mathcal{G}$ satisfies, obviously, all the conditions of Result~\ref{R: Xu 2010}. Therefore,
$\mathcal{G}$ is normal and the normality of $\mathcal{F}$ then follows. 
\smallskip

The Montel--Carath{\'e}odory Theorem was generalized to higher dimensions by Dufresnoy~\cite{Dufresnoy}.
To state the Montel--Carath{\'e}odory Theorem in higher dimensions, 
we need to introduce some essential notions. To this end, we   
fix a system of homogeneous coordinates $w= [w_0: w_1:\cdots: w_n]$ on the 
$n$-dimensional complex projective space $\mathbb{P}^n$, $n\geq 1$.
A hyperplane $H$ in $\mathbb{P}^n$ can be given by
\begin{equation}\label{Eq: hyperplane}
  H := 
  \left\{[w_0: w_1:\cdots: w_n]\in\mathbb{P}^n\,\big| 
  \sum\nolimits_{l=0}^{n}a_{l}w_{l}=0\right\},
\end{equation}
where $(a_0, a_1,\dots, a_n)\in\mathbb{C}^{n+1}$ is a non-zero vector. 
Recall that for any collection of $q (\geq n+1)$ hyperplanes $H_1,\dots, H_q$, we say that these hyperplanes  
are in {\it general position} if for any arbitrary subset
$R\subset\{1,\dots, q\}$\,---\,with cardinality $|R|=n+1$\,---\,the 
intersection $\bigcap_{k\in R}H_k=~\emptyset$.%
\smallskip

We are now in a position to state the Montel--Carath\'eodory Theorem in higher dimensions:
A family of holomorphic $\mathbb{P}^n$-valued mappings on a domain $D\subseteq\mathbb{C}^m$
is normal if this family omits $2n+1$ hyperplanes in general position in $\mathbb{P}^n$.
\smallskip
 
A natural question arises immediately in connection with the above result:
\begin{itemize}
 \item[$(\bullet)$] Does a version of Result~\ref{R: Xu 2010}\,---\,that improves the Montel--Carath\'eodory 
 Theorem\,---\,holds true for any family of holomorphic mappings
 from a planar domain into $\mathbb{P}^n$? 
\end{itemize}
Concerning Result~\ref{R: Xu 2010}: note that the family $\mathcal{F}$ in the statement of 
the Result~\ref{R: Xu 2010} is precisely a family of holomorphic mappings 
from $D$ into $\mathbb{P}^1$.  We recall, at this point, that a holomorphic mapping from a Riemann surface into a 
complex manifold is called a {\it holomorphic curve}.

\smallskip

We call the following $n+1$ hyperplanes
\begin{equation*}
 H_l:=\left\{[w_0: w_1:\cdots: w_n]\in\mathbb{P}^n\,\big| 
 w_{l-1}=0\right\};\ l=1,\dots, n+1, 
\end{equation*} 
the {\it coordinate hyperplanes} of $\mathbb{P}^n$. Note that these $n+1$ coordinate hyperplanes are in 
general position in $\mathbb{P}^n$. In the set up of Result~\ref{R: Xu 2010}, 
the values $0$ and $\infty$ can be viewed as the coordinate hyperplanes of the 
projective space $\mathbb{P}^1$. With these observations, we give an answer to
$(\bullet)$  by introducing our first theorem. Once we examine the Condition~$(b)$ in our theorem
in the context of Result~\ref{R: Xu 2010}, it will be clear that the latter is a special case of the following:

\begin{theorem}\label{T: Main Xu shared hyperplane}
 Let $\mathcal{F}$ be a family of holomorphic curves on a planar domain $D$ into $\mathbb{P}^n$.  
 Let $H_1,\dots, H_{2n+1}$ be $2n+1$ hyperplanes in general position in $\mathbb{P}^n$. 
 Assume that the first $n+1$ of these hyperplanes, i.e.,   
 $H_1,\dots, H_{n+1}$, are the coordinate hyperplanes of $\mathbb{P}^n$. 
 Suppose that
 \begin{enumerate}
  \item[$(a)$] for each pair $f, g\in \mathcal{F}$; $f^{-1}(H_k)=g^{-1}(H_k)$  for each $k\in\{1,\dots, n+1\}$;
  \item[$(b)$] for each compact set $K\subset D$ there exists a finite number $M \equiv M(K)>0$ such that
  for each $f\in\mathcal{F}$ and each $l\in\{0,\dots, n\}$ for which
  $U_{f,\,l} := \big\{z\in D : f(z)\in \{[w_0:\dots:w_n]\,|\,w_l\neq 0\}\big\}\neq \emptyset$,
  the inequality
  \[
    \big|(\widetilde{{}^l\nf}^*H_k)^{(q)}(z)\big|_{_{1\leq q\leq n}}\leq M \quad  \forall z\in f^{-1}(H_k)\cap K
  \] 
  holds for each $k\in\{n+2, \dots, 2n+1\}$, where
  $\widetilde{{}^l\nf}=({}^l\nf_0,\dots,{}^l\nf_n)$ is the reduced representation of 
   $f$ on the subset $U_{f,\l}\subset D$ such that ${}^l\nf_l\equiv 1$ on $U_{f,\,l}$. 
 \end{enumerate}
 Then the family $\mathcal{F}$ is normal.
\end{theorem}

\begin{remark}\label{R: why general}
The Montel--Carath{\'e}odory Theorem for $\mathbb{P}^n$-valued 
holomorphic curves is not deducible from Theorem~\ref{T: Main Xu shared hyperplane} when $n\geq 2$.
Nevertheless, there are examples of families of $\mathbb{P}^n$-valued 
holomorphic curves, $n\geq 2$, for which no information about their normality can be deduced from the
Montel--Carath{\'e}odory Theorem, but we know these families to be normal owing to 
Theorem~\ref{T: Main Xu shared hyperplane}\,---\,see Example~\ref{Ex: Main Thm works}, for instance.
We must also remark that while the need to control the specific reduced representations 
$\widetilde{{}^l\nf}$ in Theorem~\ref{T: Main Xu shared hyperplane} may appear to be technical, 
Example~\ref{Ex: b is essential} below reveals that this is essential.
\end{remark}
 
The object $\widetilde{{}^l\nf}^*H_k$ refers to a holomorphic function on $U_{f,\,l}$ whose
precise definition is given in Section~\ref{S:notions} (see \eqref{eq: f*h}). 
But to get a sense of these functions (and what Condition~($b$) says) consider the case $n=1$.
Then a hyperplane $H$ is just a point $\alpha\in\mathbb{C}\cup\{\infty\}$. 
Also, for a non-constant holomorphic curve $f: D\rightarrow\mathbb{P}^1$, we classically understand $f$ to
be equal to a ratio $f=f_1/f_0$, where $f_0, f_1 (\not\equiv 0)$ are holomorphic functions on $D$
having no common zeros, and where the points in $f^{-1}\{\infty\}$ (if any) are the poles of $f_1/f_0$
or, equivalently, the zeros of $f_0$. In this
case $U_{f,\,0} = D\setminus f_0^{-1}\{0\}$, the reduced reperesentation $\widetilde{{}^0\nf}$ referenced 
by Condition~($b$) is simply $(1, f_1/f_0)$.  Furthermore (see \eqref{eq: f*h} below),
\[
  \widetilde{{}^0\nf}^*H := f_1/f_0 - \alpha \text{ (on $U_{f,\,0} = D\setminus f_0^{-1}\{0\}$) if $\alpha\neq \infty$.}
\]
Arguing further: in the final analysis we deduce that the zeros of the holomorphic function 
\begin{itemize}
  \item $\widetilde{{}^0\nf}^*H(z)$ are the zeros of (the meromorphic function) $f-\alpha$ in $D$
  when $\alpha\neq \infty$; and
  \item $\widetilde{{}^1\nf}^*H(z)$ are the poles of (the meromorphic function) $f$ in $D$
  when $\alpha = \infty$.
\end{itemize}   
The notation $(\widetilde{g_f}^*H_k)^{(q)}$\,---\,$g_f$ denoting any of the relevant reduced
representations\,---\,represents
the $q$-th derivative of $\widetilde{g_f}^*H_k$. It is now not hard to  
notice that Condition~($b$) in the statement of Theorem~\ref{T: Main Xu shared hyperplane} is 
weaker than the Condition~($b$) in the statement of Result~\ref{R: Xu 2010}. 
In fact, if the zeros of $\widetilde{g_f}^*H_k$ are
of multiplicity at least ($n+1$) (with $g_f$ as above and $n+2\leq k\leq 2n+1$)
then Condition~($b$), in the statement of Theorem~\ref{T: Main Xu shared hyperplane},
automatically holds true. With this observation, we have a corollary  
with a more attractive statement:

\begin{corollary}\label{Cor: Main Xu shared hyperplane}
 Let $\mathcal{F}$ be a family of holomorphic curves on a planar domain $D$ into $\mathbb{P}^n$. 
 Let $H_k$, $k=1, \dots, 2n+1$, be hyperplanes  as in the 
 statement of Theorem~\ref{T: Main Xu shared hyperplane}.  
 Suppose that
 \begin{enumerate}
  \item[$(a)$] for each pair $f, g\in \mathcal{F}$; $f^{-1}(H_k)=g^{-1}(H_k)$  for each $k\in\{1,\dots, n+1\}$;
  \item [$(b)$] for each $f\in\mathcal{F}$ and each $k\in\{n+2,\dots, 2n+1\}$; 
   $f$ intersects $H_k$ with multiplicity at least $n+1$.
 \end{enumerate}
 Then the family $\mathcal{F}$ is normal.
\end{corollary}

We now ask a related question: whether one can replace Condition~$(b)$ in the statement of 
Corollary~\ref{Cor: Main Xu shared hyperplane} by the condition ($\star$), given below, and yet yields 
the same conclusion.
\begin{itemize}
\item[$(\star)$]
the volumes of $f^{-1}(H_k)$, viewing $f^{-1}(H_k)$ as divisors, are locally uniformly bounded.
\end{itemize}
Example~\ref{Ex: vol of divisor} below will show that this replacement
is not the right  replacement.  We therefore ask: how close a condition to Condition ($\star$)
would give the same conclusion as in the Corollary~\ref{Cor: Main Xu shared hyperplane}?
In order to answer this question we state 
our next result which is a generalization of the Montel--Carath{\'e}odory Theorem
(we shall presently see the relevance, in the following theorem, of considering volumes):

\begin{theorem}\label{T: vol of divisor}
 Let $\mathcal{F}$ be a family of holomorphic curves on a planar domain $D$
 into $\mathbb{P}^n$, $n\geq 2$, and $H_1,\dots, H_{2n+1}$
 be hyperplanes in general position in $\mathbb{P}^n$. Suppose that
 \begin{enumerate}
  \item [$(a)$] for each $f\in\mathcal{F},\, f(D)\cap H_k=\emptyset$ for all $k = 1,\dots, n+1$;
 \end{enumerate}
 and that there exists an integer $t$, $n+2\leq t< 2n+1$, such 
 that
 \begin{enumerate} 
  \item [$(b)$] for each pair of holomorphic curves $f, g\in \mathcal{F}$, 
  $f^{-1}(H_k)=g^{-1}(H_k)$ for each $n+2\leq k\leq t$;
  \item [$(c)$] for each $f\in\mathcal{F}$ and  $k\in\{t+1,\dots, 2n+1\}$, 
  $f(D)\not\subset H_k$ and, for each compact set $K\subset D$, the volumes of $f^{-1}(H_k)\cap K$, viewing 
  $f^{-1}(H_k)$ as divisors, are uniformly bounded (i.e., for $k= t+1,\dots, 2n+1$, and all $f\in\mathcal{F}$).
 \end{enumerate}
 Then the family $\mathcal{F}$ is normal.
\end{theorem}

We clarify here that the volume of $f^{-1}(H_k)\cap K$\ is identified with the 
cardinality of the set of zeros\,---\,counting with multiplicity\,---\,of $\widetilde{f}^*H_k$ in $K$, for some 
reduced representation $\widetilde{f}$ of $f$. 
We end this section 
with a brief explanation 
for the above-mentioned 
assertion that Theorem~\ref{T: vol of divisor} is a generalization of the Montel--Carath{\'e}odory Theorem 
(for holomorphic curves). 
To this end, suppose a family $\mathcal{F} $ of $\mathbb{P}^n$-valued holomorphic curves on a planar domain 
$D$ omits $2n+1$ hyperplanes
in general position\,---\,say $H_1, \dots, H_{2n+1}$ in $\mathbb{P}^n$. This means that 
\begin{equation}\label{Eq: omit hyp}
f^{-1}(H_k)=\emptyset \text{ for 
all } f\in\mathcal{F} \text{ and } k=1, \dots, 2n+1.
\end{equation} 
We now fix $t= 2n$~\!($\geq n+2$, since $n\geq 2$)  
and notice, easily, that $\mathcal{F}$ satisfies Conditions~$(a)$ and~$(b)$ of  Theorem~\ref{T: vol of divisor}.
We next infer, from \eqref{Eq: omit hyp}, that
$f(D)\not\subset H_{2n+1}$ and the volume of $f^{-1}(H_{2n+1})\cap K$ is $0$ for each $f\in\mathcal{F}$ 
and each compact subset $K$ of $D$.
Thus, Condition~$(c)$ of  Theorem~\ref{T: vol of divisor}
is also satisfied and hence $\mathcal{F}$ is normal. Finally, to make the case that
Theorem~\ref{T: vol of divisor} is a generalization of the Montel--Carath{\'e}odory Theorem,
we must produce an example of a planar domain $D$, a family $\mathcal{F}$, and hyperplanes
$H_1,\dots, H_{2n+1}$ as in Theorem~\ref{T: vol of divisor} (and satisfy the conditions thereof) that do
\textbf{not} satisfy the
hypothesis of Montel--Carath{\'e}odory Theorem. Interestingly, the example referred to in
Remark~\ref{R: why general}\,---\,see Section~\ref{S: Examples} for details\,---\,serves the latter
purpose as well.    
\medskip

\section{Basic notions}\label{S:notions}

This section is devoted to elaborating upon concepts and terms that made an
appearance in Section~\ref{S:intro}, and to introducing certain basic notions needed in our
proofs. In this section, $D\subseteq\mathbb{C}$ will always denote a planar domain.
\smallskip

Let $f: D\rightarrow \mathbb{P}^n$ be a holomorphic mapping.
Fixing a system of homogeneous coordinates on the 
$n$-dimensional complex projective space
$\mathbb{P}^n$, for
each $a\in D$, we have a holomorphic map 
$\widetilde{f}(z):= (f_0(z), f_1(z),\dots, f_n(z))$ on some neighborhood $U$ of $a$
such that $\{z\in U\,|\,f_0(z)= f_1(z)=\cdots= f_n(z)=0\}=\emptyset$ and
$f(z) = [f_0(z): f_1(z):\cdots: f_n(z)] \ \text{ for each } z\in U$.
We shall call any such holomorphic map $\widetilde{f} : U\rightarrow \mathbb{C}^{n+1}$
a {\it reduced representation} (or an admissible representation) of $f$ on $U$.  A holomorphic 
mapping $f: D\rightarrow\mathbb{P}^n$ is called a {\it holomorphic curve} of $D$ into $\mathbb{P}^n$.
Note that, as $\dim_{\mathbb{C}}D = 1$, every holomorphic curve on $D$ into
$\mathbb{P}^n$ has a reduced representation globally on $D$. However, we will find it useful
to work with local reduced representations as well.  
\smallskip
 
For a fixed system of homogeneous coordinates on $\mathbb{P}^n$ we set
\[
V_l:=\{[w_0:\dots:w_n]\,|\, w_l\neq 0\},\quad \text{ for } l= 0, \dots, n.
\]
Then every point $a\in D$ has a neighborhood $U\ni a$ such that $f(U)\subset V_l$ for some $l$, and
$f$ has a reduced representation
$\widetilde{f}=(f_0,\dots, f_{l-1}, 1, f_{l+1},\dots, f_n)$ 
on $U$, where $f_0,\dots,f_{l-1}, f_{l+1},\dots, f_n$ are holomorphic functions on $U$.
\smallskip

Let $f\not\equiv 0$ be a holomorphic function on $D$. 
For a point $a\in D$, the number
\begin{equation*}
  \nu_{f}(a):= \begin{cases}
  0,\quad \ \text{if } f(a)\neq 0\\
  m,\quad \text{if } f \text{ has a zero of multiplicity } m \text{ at } a
  \end{cases} 
\end{equation*} 
is said to be the {\it zero-multiplicity of $f$ at $a$}. An integer-valued function $\nu: D\rightarrow \mathbb{Z}$ 
is called a {\it divisor} on $D$ if for each point $a\in D$ there exist holomorphic functions 
$g\not\equiv 0$ and $h\not\equiv 0$ in a neighborhood $U$ of $a$ such that 
$\nu(z)= \nu_{g}(z)-\nu_{h}(z)$ for all  $z\in U$. A divisor $\nu$ on $D$ is said to be {\it non-negative} 
if $\nu(z)\geq 0$ for all $z\in D$.  We define the support $\mathsf{supp}\,\nu$ of a non-negative divisor
$\nu$ on $D$ by
\begin{equation*}
 \mathsf{supp}\,\nu:= \{z\in D\,|\,\nu(z)\neq 0 \}.
\end{equation*}

Let $H$ be a hyperplane in $\mathbb{P}^n$ as defined in \eqref{Eq: hyperplane}.
Let $f: D\rightarrow \mathbb{P}^n$ be a holomorphic curve such that
$f(D)\not\subseteq H$. Under this condition it is possible to define a
divisor on $D$ that is canonically associated with the pair $(f, H)$, which we shall denote by $\nu(f, H)$.
(We have not defined divisors on general complex manifolds, but we observe that $H$ determines a divisor on
$\mathbb{P}^n$ and the divisor we shall define is the pullback of the latter.) To do so,   
consider 
any $a\in D$, take a reduced representation 
$\widetilde{f}:= (f_0, f_1,\dots, f_n)$ of $f$ on a neighborhood $U$ of $a$, 
and consider the following holomorphic function 
\begin{equation}\label{eq: f*h}
 \widetilde{f}^*H:= a_0f_0 + a_1f_1 +\dots+ a_nf_n.
 \end{equation}
It follows from the definition of a reduced representation that {the values of
$\nu_{\widetilde{f}^*H}$ agree, on an appropriate neighborhood of $a$, for any two
reduced representations of $f$ around $a$. It is now easy
to check that if one defines $\nu(f, H)$ by
\[ 
  \left.\nu(f, H)\right|_U(z):= \nu_{\widetilde{f}^*H}(z), \; \; z\in U,
\]
then $\nu(f, H)$ is well defined globally to give a divisor on $D$.
\smallskip

We say that the {\it holomorhic mapping $f$ intersects $H$ with multiplicity at least $m$
on $D$}, $m\in \mathbb{Z}_+$, if $\nu(f, H)\geq m$ for all $z\in\mathsf{supp}\,\nu(f, H)$.
Furthermore, we say that
$f$ {\it intersects $H$ with multiplicity $\infty$ on $D$} if 
either $f(D)\subset H$ or $f(D)\cap H=\emptyset$.
Here we remark that one can identify the support $\mathsf{supp}\,\nu(f, H)$ of $\nu(f, H)$ 
by $f^{-1}(H_k)$ which is again identified locally with the set of zeros of 
the function $\widetilde{f}^*H$ ignoring multiplicities.
\smallskip

We now give the definition that is central to the discussion in Section 1.
Let $M$ be a compact connected Hermitian manifold.
The space $\operatorname{Hol}(D, M)$ of holomorphic 
mappings from $D$ into $M$ is endowed with the compact-open topology.
\begin{definition}\label{Def:normal}
  A family $\mathcal{F}\subset\operatorname{Hol}(D, M)$ is said to be {\it normal} if 
  $\mathcal{F}$ is relatively compact in $\operatorname{Hol}(D, M)$.
\end{definition}
 

\section{Some Examples}\label{S: Examples}

We now provide the examples alluded to in Section~\ref{S:intro}.

\begin{example}\label{Ex: b is essential}
The ``technical restriction'' imposed on the reduced representations of holomorphic curves of the family $\mathcal{F}$ 
in the statement of Condition~$(b)$ of Theorem~\ref{T: Main Xu shared hyperplane} is \textbf{essential}. 
More precisely: the conclusion of 
Theorem~\ref{T: Main Xu shared hyperplane} need not follow if the bounds in Condition~$(b)$
apply to some reduced representation of $f\in\mathcal{F}$ but which are \textbf{not} 
the reduced representations $\widetilde{{}^l\nf}$ of Condition~$(b)$.
\end{example}

Let $D=\mathbb{C}$ and $\mathcal{F}_1=\{f_j(z)\,|\, j\in \mathbb{Z}_+\}$, where 
$f_j:D\rightarrow \mathbb{P}^2$ is defined by
\begin{equation*}
 f_j(z):= [jz^{2}:1:1].
\end{equation*}
Let $H_1, H_2$ and $H_3$ be the coordinate hyperplanes of $\mathbb{P}^2$ given by
$$H_l:=\{[w_0: w_1:w_2]\,\big|\, w_{l-1}=0\},\ l=1, 2, 3.$$ 
Let 
\begin{align*}
 H_4&:=\{[w_0: w_1: w_2]\,|\, w_0+w_1+w_3=0\};  \text{ and }\\
 H_5&:=\{[w_0: w_1: w_2]\,|\, w_0+ 2w_1+ 3w_3=0\}.
\end{align*}
Clearly, hyperplanes $H_1, \dots, H_5$ are in general position. 
It is easy to see that $f_j^{-1}(\{H_1\})=\{0\}$ and $f_j^{-1}(\{H_2\})=\emptyset=f_j^{-1}(\{H_3\})$ for all $j$. 
Suppose we consider the following reduced representation of $f\in\mathcal{F}_1$\,---\,where $f_l\not\equiv1$ for all $0\leq l\leq 2$:
\begin{equation*}
 \widetilde{f}_j(z):= (z^{2}, 1/j, 1/j); \  z\in\mathbb{C} \text{ and }  \mathbb{Z}_+\ni j\geq 2.
\end{equation*}
Then $\widetilde{f}_j^*H_4(z)=z^2+2/j$, and $\widetilde{f}_j^*H_5(z)=z^2+5/j$ for all $z\in D$,  $j\geq 2$.
It is easy to see that for $k=4, 5$, we have
\[
 |(\widetilde{f}_j^*H_k)^{(q)}(z)|\leq 2 \text{ for } z\in f_j^{-1}(H_k),
 \text{ and } q=1, 2.
\]
Thus, all the conditions in the statement of Theorem~\ref{T: Main Xu shared hyperplane}  
hold true\,---\,with the only exception that the estimate stated in Condition~($b$) applies
a reduced representation of $f\in\mathcal{F}_1$ such that 
$f_l\not\equiv 1$ for all $0\leq l\leq n$ (as opposed to the
reduced representations identified in Theorem~\ref{T: Main Xu shared hyperplane}).  
However, the family $\mathcal{F}_1$ is not normal. \hfill $\blacktriangleleft$
\smallskip

Let us look back at the family $\mathcal{F}_1$. Clearly, from the discussion in
Example~\ref{Ex: b is essential}, we expect Condition~($b$) of 
Theorem~\ref{T: Main Xu shared hyperplane} to fail for $\mathcal{F}_1$. Let us examine precisely
how it fails.
To this end, observe that
\[
  \big\{z\in D : f_{j}(z)\in \{[w_0: w_1: w_2]\,|\,w_l\neq 0\}\big\} = \mathbb{C} \text{ for $l=1, 2$ and
  for all $j\in \mathbb{Z}_+$.}
\]
Consider the reduced representations
$\widetilde{{}^1\nf}_j(z)= 
\widetilde{{}^2\nf_j}(z):=(jz^2, 1, 1)$ of $f_j$ in $\mathbb{C}$, $j\in\mathbb{Z}_+$.
Let us use the notation $\widetilde{f}_j$ for $\widetilde{{}^1\nf}_j$ and $\widetilde{{}^2\nf_j}$.
Then for $k=4, 5$, we have 
\[
 |(\widetilde{f}_j^*H_k)^{(q)}(z)|\rightarrow \infty, \text{ as } j\rightarrow \infty \text{ for } z\in f_j^{-1}(H_k) 
 \text{ and } q= 1, 2.
\]
This implies the (anticipated) failure of Condition~$(b)$ in the statement of Theorem~\ref{T: Main Xu shared hyperplane}
for the family $\mathcal{F}_1$.

\smallskip

\begin{example}\label{Ex: cor is essential}
 The condition that holomorphic curves $f\in\mathcal{F}$ intersect $H_k$ with multiplicity at least $n+1$
 in Condition $(b)$ of Corollary~\ref{Cor: Main Xu shared hyperplane} is essential, when $n=1$. 
\end{example}

Let $D=\mathbb{C}$, and $\mathcal{F}_2=\{f_j(z)\,|\, j\in \mathbb{Z}_+\}$, where 
$f_j:D\rightarrow \mathbb{P}^1$ is defined by
\begin{equation*}
f_j(z):= [jz:1].
\end{equation*}
Let $H_1, H_2$ be the coordinate hyperplanes of $\mathbb{P}^1$ such that 
$$H_l:=\{[w_0: w_1]\,\big|\, w_{l-1}=0\},\ l=1, 2.$$
Let $H_3:=\{[w_0:w_1]\,|\, w_0-w_1=0\}$. The hyperplanes $H_1, H_2, H_3$ are in general position. 
Clearly, $f_j^{-1}(\{H_1\})=\{0\}$ and $f_j^{-1}(\{H_2\})=\emptyset$ 
for all $j$. Whereas $\widetilde{f}_j^*H_3$ has a zero of multiplicity $1$ for each $j$, where 
$\widetilde{f}_j(z):=(z, 1/j)$, $z\in\mathbb{C}$, for all $j\in\mathbb{Z}_+$. 
Thus, all the conditions in the statement of Corollary~\ref{Cor: Main Xu shared hyperplane} 
hold true except for $(b)$. However, the family $\mathcal{F}_2$ fails to be normal.
\hfill $\blacktriangleleft$
\smallskip

\begin{example}\label{Ex: vol of divisor}
The number $n + 1$ is sharp in the statement of Condition $(a)$ of Theorem~\ref{T: vol of divisor}.
Specifically: The conclusion of Theorem~\ref{T: vol of divisor} need not hold true if at most $n$ 
hyperplanes are omitted by the family $\mathcal{F}$.
\end{example}

Let $D=\mathbb{C}$, and $\mathcal{F}_3=\{f_j(z)\,|\, j\in \mathbb{Z}_+\}$, where 
$f_j:D\rightarrow \mathbb{P}^n$ is defined by
\begin{equation*}
f_j(z):= [1:2:\dots:n: jz].
\end{equation*}
Let $H_1,\dots, H_{n+1}$ be coordinate hyperplanes of $\mathbb{P}^n$ such that 
$$H_l:=\{[w_0: w_1:\dots w_n]\,\big|\, w_{l-1}=0\},\ l=1, \dots, n+1.$$
Let $H_{n+2},\dots, H_{2n+1}$ be any hyperplanes such that $H_1,\dots, H_{2n+1}$ are in general position. 
Clearly, for each $j$,
\begin{itemize}
 \item $f_j^{-1}(\{H_{l}\})=\emptyset$  for $l=1,\dots, n$; and
 \item $f_j^{-1}(\{H_{n+1}\})=\{0\}$. 
\end{itemize}
It is easy to see that for each $j$ and each $k\in\{n+2, \dots, 2n+1\}$, $f_j^{-1}(H_k)$ consists of 
only one point in $D$. Thus, $\mathcal{F}_3$
satisfies all the conditions\,---\,with the understanding that $k$ is limited to $k = 1, \dots, n$
in the condition ($a$)\,---\,in the 
statement of Theorem~\ref{T: vol of divisor}. However, the family $\mathcal{F}_3$ is not normal.
 \hfill $\blacktriangleleft$
\smallskip

Our final example is the one alluded to in Remark~\ref{R: why general}. Note that the same example also serves
as an example that does not satisfy the hypothesis of the Montel--Carath{\'e}odory Theorem but does
satisfy the conditions of Theorem~\ref{T: vol of divisor}.

\begin{example}\label{Ex: Main Thm works}
There exist a family of $\mathbb{P}^n$-valued holomorphic curves, $n=2$, for which the 
Montel--Carath{\' e}odory Theorem does not give any information 
but whose normality is deducible either 
from Theorem~\ref{T: Main Xu shared hyperplane} or Theorem~\ref{T: vol of divisor}.
\end{example}

Let $D=\mathbb{C}$, $n=2$, and $\mathcal{F}_4=\{f_j: j\in \mathbb{Z}_+\}$, where 
$f_j:\mathbb{C}\rightarrow \mathbb{P}^2$ is defined by 
\begin{equation*} 
f_j(z):=[je^z: 1: -1].
\end{equation*}
Let $H_1, H_2, H_3$ be the coordinate hyperplanes in $\mathbb{P}^2$ given by 
\begin{equation*}
H_l:=\{[w_0:w_1:w_2]\in\mathbb{P}^2: w_{l-1}=0\}, \ l =1, 2, 3.
\end{equation*} 
Let 
\begin{align*}
 H_4&:=\{[w_0: w_1: w_2]\,|\, w_0+w_1+w_3=0\};  \text{ and }\\
 H_5&:=\{[w_0: w_1: w_2]\,|\, w_0+ 2w_1+ 3w_3=0\}.
\end{align*} 
Clearly, hyperplanes $H_1,\dots, H_5$ are in general position.
Since $\mathcal{F}_4$ consists of non-constant holomorphic curves $f_j:\mathbb{C}\rightarrow \mathbb{P}^2$,
$\mathcal{F}_4$ cannot omit $5 (=2n+1)$ hyperplanes in general position in $\mathbb{P}^2$. Therefore, the 
Montel--Carath\'eodory theorem cannot give any information about the normality of $\mathcal{F}_4$.
\smallskip

We now show, by using Theorem~\ref{T: Main Xu shared hyperplane} 
and Theorem~\ref{T: vol of divisor}, that the family $\mathcal{F}_4$ is normal.
Clearly, $f_j^{-1}(H_k)=\emptyset \ \text{ for all } j\in\mathbb{Z}_+ \text{ and } k=1, 2, 3, 4.$ Thus,
\begin{itemize}
\item Condition~$(a)$ of Theorem~\ref{T: Main Xu shared hyperplane} is satisfied.
\item Conditions~$(a)$ $\&$ $(b)$ of Theorem~\ref{T: vol of divisor} are satisfied after fixing $t=4$.   
\end{itemize}
Since $f_j^{-1}(H_4)=\emptyset$, the condition~$(b)$ of Theorem~\ref{T: Main Xu shared hyperplane}
hold true for $H_4$. We now show that 
Condition~$(b)$ of Theorem~\ref{T: Main Xu shared hyperplane} holds true for $H_5$:\\
We first consider the following reduced representation of $f_j$
$$\widetilde{{}^0\nf_j}(z):=(1, e^{-z}/j, -e^{-z}/j), \quad z\in\mathbb{C}.$$
Then, we have 
$${\widetilde{{}^0\nf_j}^*H_5(z)=1-{e^{-z}}/{j}}; \
{(\widetilde{{}^0\nf_j}^*H_5)'(z)={e^{-z}}/{j}}; \text{ and }
{(\widetilde{{}^0\nf_j}^*H_5)''(z)={-e^{-z}}/{j}}.$$ For any compact set $K\subset \mathbb{C}$, there exists an $R\equiv R(K)>0$ 
such that $K\subset B(0, R)$, where $B(0, R)$ is the open disc centered at the origin with radius $R$. Therefore, we have 
$$|(\widetilde{{}^0\nf_j}^*H_5)^{(q)}(z)|=|{e^{-z}}/{j}|\leq e^R \quad \text{for all } j\in\mathbb{Z}_+, \ z\in\overline{B(0, R)}
\ \text{and } q=1, 2.$$
Choose $M\geq e^R$, whence Condition~$(b)$ of Theorem~\ref{T: Main Xu shared hyperplane} is satisfied in this case.
We now consider the following reduced representation of $f_j$
$$\widetilde{{}^1\nf_j}(z):=(je^{z}, 1, -1), \quad z\in\mathbb{C}.$$
Then, we have 
 $$\displaystyle{\widetilde{{}^1\nf_j}}^*H_5(z)=j e^z-1 \text{ and }
(\widetilde{{}^1\nf_j}^*H_5)'(z)= je^z = (\widetilde{{}^1\nf_j}^*H_5)''(z).$$ Clearly, for each $z\in f_j^{-1}(H_5)$,
we have $e^z=1/j$.  This implies that
\begin{equation*}
 |(\widetilde{{}^1\nf_j}^*H_5)^{(q)}(z)|=1 \quad \text{ for all } z\in f_j^{-1}(H_5) \text{ and } q=1, 2.
\end{equation*}
Similar calculation holds for the reduced representation $\widetilde{{}^2\nf_j}(z):=(-je^{z}, -1, 1)$.
Thus, Condition~$(b)$ of Theorem~\ref{T: Main Xu shared hyperplane} is satisfied for $H_5$ and hence $\mathcal{F}_4$ is 
normal.
\smallskip

Now we show that Condition~$(c)$ of Theorem~\ref{T: vol of divisor} is also satisfied. 
Consider the reduced representations $\widetilde{f_j}(z):=(e^z, 1/j, -1/j)$  of $f_j$ on $\mathbb{C}$, $j\in \mathbb{Z}_+$.
Clearly, $f_j(\mathbb{C})\not\subset H_5$ and for any compact subset $K\subset \mathbb{C}$ the volumes of 
$f_j^{-1}(H_5)\cap K$ are uniformly bounded, which is evident from the following observation:
For a closed ball $\overline{B(0, R)}$ of radius $R>0$ centered at $0$, the cardinality of the set of zeros of 
$\widetilde{f_j}^*H_5$ is bounded by $\lceil{{R}/{\pi}}\rceil+1$ for all $j\in\mathbb{Z}_+$. Where, $\lceil{{\bullet}}\rceil$
is the ceiling function.
Thus, Condition~$(c)$ of Theorem~\ref{T: vol of divisor} is also satisfied, therefore $\mathcal{F}_4$ is normal. 
Hence we conclude that 
$\mathcal{F}_4$ is normal by Theorem~\ref{T: Main Xu shared hyperplane} 
as well as by  Theorem~\ref{T: vol of divisor} but we cannot conclude the normality of $\mathcal{F}_4$ 
from the Montel--Carath\'eodory theorem. 
\medskip

\section{Essential lemmas}\label{S: lemmas}
In order to prove our theorems, we need to state certain known results and prove some essential lemmas.
\smallskip

One of the well-known tools in the theory of normal families in one complex variable is Zalcman's lemma.
Roughly speaking, it says that the failure of normality implies that a certain kind of infinitesimal 
convergence must take place. The higher dimensional analogue of Zalcman's rescaling lemma 
is as follows:

\begin{lemma}[{\cite[Theorem 3.1]{AladroKrantz}, \cite[Corollary 2.8]{Thai Trang Huong 03}}]\label{L: AK}
  Let $M$ be a compact complex space, and $\mathcal{F}$ a family of holomorphic mappings 
  from a domain $D\subset \mathbb{C}^m$ into $M$. The family $\mathcal{F}$ is not normal 
  if and only if there exist
  \begin{enumerate}
    \item[$(a)$] a point $\xi_0\in D$ and a sequence $\{\xi_j\}\subset D$ such that  $\xi_j\rightarrow \xi_0$;
    
    \item[$(b)$] a sequence $\{f_j\}\subset \mathcal{F}$;
    
    \item[$(c)$] a sequence $\{r_j\}\subset \mathbb{R}$ with $r_j>0$ and $r_j\rightarrow 0$  
  \end{enumerate} 
  such that $ h_j(\zeta):= f_j(\xi_{j}+r_{j}\zeta)$, where $\zeta\in\mathbb{C}^m$ satisfies
  $\xi_{j}+r_{j}\zeta\in D$, converges uniformly on compact subsets of $\mathbb{C}$ to a 
  non-constant holomorphic mapping $h: \mathbb{C}\rightarrow M.$
\end{lemma}

\begin{remark}
  We remark that the result of Aladro--Krantz in \cite{AladroKrantz} has a weaker hypothesis
  than in Lemma~\ref{L: AK}. In the case where $M$ is {\it non-compact}, there is a case missing from
  their analysis. The arguments needed in this case were provided by
  \cite[Theorem 2.5]{Thai Trang Huong 03}. At any rate, Lemma~\ref{L: AK} is the version of the
  Aladro--Krantz theorem that we need.
\end{remark}

\subsection{Results on Nevanlinna theory}
Let $\nu$ be a non-negative divisor on $\mathbb{C}$. Given a  $p\in\mathbb{Z}_+\cup\{\infty\}$, we define 
the {\it truncated counting function} of $\nu$\,---\,multiplicities are truncated by $p$\,---\,by  
\[
 N^{[p]}(r, \nu):=\int_{1}^{r} \frac{n^{[p]}_{\nu}(t)}{t}dt \qquad (1< r<+\infty),
\]
where $n^{[p]}_{\nu}(t):=\sum_{|z|\leq t}\min\{\nu(z), p\}$.
\smallskip

Let $f:\mathbb{C}\rightarrow\mathbb{P}^n$ be a holomorphic curve, and $H$ be a hyperplane in $\mathbb{P}^n$. 
If  $f(\mathbb{C})\not\subseteq H$ then we define 
\[
 N^{[p]}_f(r, H):= N^{[p]}\left(r, \nu(f, H)\right); \quad \text { and } \quad N_f(r, H):= N^{[+\infty]}_f(r, H).
 \] 

Let $f:\mathbb{C}\rightarrow\mathbb{P}^n$ be a holomorphic curve. 
Let $\widetilde{f}=(f_0,\dots, f_n)$\,---\,for arbitrarily fixed
homogeneous coordinates on $\mathbb{P}^n$\,---\,be a reduced representation of $f$. Then we
define the {\it characteristic function} $T_f(r)$ of $f$ by
\begin{equation*}
 T_f(r):= \frac{1}{2\pi}\int_{0}^{2\pi} \log \|f(re^{i\theta})\| d\theta -\log \|f(0)\|,
\end{equation*}
where $\|f\|:=(|f_0|^2+\dots+|f_n|^2)^{1/2}$. 
We now state the following version of the First  Main Theorem
in Nevanlinna Theory which relates the characteristic function and the counting function. 

\begin{fmt}[{\cite[Theorem 2.3.31]{Noguchi Winkelmann 2014}}]\label{L: FMT}
 Let $f:\mathbb{C}\rightarrow\mathbb{P}^n$ be a holomorphic curve and $H$ be a hyperplane such that 
 $f(\mathbb{C})\not\subseteq H$. Then
 \begin{equation*}
  N_f(r, H)\leq T_f(r) +O(1) \quad \text {for all } r>1.
 \end{equation*}
\end{fmt}

Generalizing H. Cartan's work, Nochka established the 
following Second Main Theorem of Nevanlinna theory for $s$-nondegenerate holomorphic 
curves into $\mathbb{P}^n$ (see \cite{Nochka 83}).
Following Nochka, a holomorphic curve $f:\mathbb{C}\rightarrow\mathbb{P}^n$ is said to be {\it $s$-nondegenerate}, 
$0\leq s\leq n$, if the dimension of the smallest linear subspace containing the image $f(\mathbb{C})$ is $s$. 

\begin{smt}[{\cite{Nochka 83}\cite[Theorem 4.2.11]{Noguchi Winkelmann 2014}}]\label{L: SMT}
 Let $f:\mathbb{C}\rightarrow\mathbb{P}^n$ be an $s$-nondegenerate holomorphic curve, 
 where $0<s\leq n$. 
 Let $H_k\not\supset f(\mathbb{C})$, $1\leq k\leq q$, be hyperplanes in general position in $\mathbb{P}^n$.
 Then the following estimate
 \begin{equation*}
  (q-2n+s-1)T_f(r) \leq \sum_{k=1}^{q}N^{[s]}_f(r, H_k) +O(\log (rT_f(r)))
 \end{equation*}  
 holds for all $r$ excluding a subset of $(1, +\infty)$  of finite Lebesgue measure.
\end{smt}

Nochka established the following Picard's type results for holomorphic curves into $\mathbb{P}^n$. 

\begin{lemma}[{\cite{Nochka 83}\cite[Corollary 4.2.15]{Noguchi Winkelmann 2014}}]\label{L: Nochka Picards type}
 Let $H_1, \dots, H_q$ be $q (\geq 2n+1)$ hyperplanes in general position in $\mathbb{P}^n$. 
 Suppose that $\{m_1,\dots, m_q\}\subset \mathbb{Z}_+\cup\{+\infty\}$. If
 \[
  \sum_{k=1}^{q}\Big(1-\frac{n}{m_k}\Big)> n+1,
 \]
 then there does not exist a non-constant holomorphic mapping $f:\mathbb{C}\rightarrow\mathbb{P}^n$  such
 that $f$ intersects $H_k$ with multiplicity at least $m_k$, $k= 1,\dots, q$.
\end{lemma}

The following lemma asserts that the logarithmic
growth of the characteristic function is attained for rational
mappings and only for them. Recall that a mapping $f:\mathbb{C}\rightarrow\mathbb{P}^n$
is {\it rational} if $f$ can be represented, in homogeneous coordinates, as $f=[f_0:\dots: f_n]$ 
with polynomial coordinates $f_l$, $l=0,\dots, n$.

\begin{lemma}[{\cite[Satz 24.1]{Stoll 54}\cite[Page 42, (7.3)]{Fujimoto 74}}]\label{L: log growth}
Let $f:\mathbb{C}\rightarrow\mathbb{P}^n$ be a holomorphic curve. Then $f$ is rational if and only if 
\[
\lim_{r\rightarrow \infty}\frac{T_f(r)}{\log r}<\infty.
\]
\end{lemma}

Yang--Fang--Pang~\cite{Yang Fang Pang 2014}\,---\,by using the above lemma\,---\,established the following result 
wherein a holomorphic curve that intersects with $2n+1$ hyperplanes reduces to 
a rational map into $\mathbb{P}^n$.

\begin{lemma}[{\cite[Lemma 3.7]{Yang Fang Pang 2014}}]\label{L: rational}
Let $f:\mathbb{C}\rightarrow \mathbb{P}^n$ be a holomorphic curve, and let $H_1,\dots, H_{2n+1}$ be 
hyperplanes in general position in $\mathbb{P}^n$. Suppose that for each 
$H_k\in\{H_1,\dots,H_{2n+1}\}$ either
 \begin{enumerate}
    \item[$(a)$] $f(\mathbb{C})\subset H_k$; or
    \item[$(b)$] $f(\mathbb{C})\not\subset H_k$ and 
     $\mathsf{supp}\,\nu(f, H_k)$ is of finite cardinality in $\mathbb{C}$.
 \end{enumerate}   
Then the map $f$ is rational.
\end{lemma}

We prove the following proposition in a similar spirit  as Lemma~\ref{L: rational}. We shall use this proposition 
in order to prove  Theorem~\ref{T: Main Xu shared hyperplane}.
\begin{proposition}\label{P: rational}
Let $f:\mathbb{C}\rightarrow \mathbb{P}^n$ be a holomorphic curve, and 
$H_1,\dots, H_{2n+1}$ be hyperplanes in genreral position in $\mathbb{P}^n$.  
Suppose that for each $H_k$, either 
\begin{enumerate}
\item[$(a)$] $f(\mathbb{C})\subset H_k$; or 
\item[$(b)$] $f(\mathbb{C})\not\subset H_k$ and one of the following holds:
\begin{enumerate} 
\item[$(i)$] $f$ intersects $H_k$ with multiplicity at least $m_k$, where 
  $m_k\in\mathbb{Z}_+\cup\{+\infty\}$; or
\item[$(ii)$]  $\mathsf{supp}\,\nu(f, H_k)$ is of finite cardinality in $\mathbb{C}$.
\end{enumerate}
\end{enumerate}
Write $\mathscr{I}:=\{k\in\{1,\dots, 2n+1\}: \mathsf{supp}\,\nu(f, H_k) 
\text{ is of finite cardinality in } \mathbb{C}\}$. Suppose 
\begin{equation*}
\sum_{k\in\{1,\dots, 2n+1\}\setminus\mathscr{I}}\frac{1}{m_k}<1.
\end{equation*}
Then $f$ is rational.
\end{proposition}
\begin{proof}
 Let $\widetilde{f}$ be a reduced representation of $f$ on $\mathbb{C}$.
 We may, without loss of generality, assume that $f$ is an $s$-nondegenerate  holomorphic curve, $0\leq s\leq n$.
 If $s=0$, then $f$ is constant and 
 hence rational. We therefore assume that $s>0$. We divide the set $\{1, \dots, 2n+1\}$ into three disjoint subsets
   \begin{align*}
    I_1&:= \{k\in\{1,\dots, 2n+1\}: f(\mathbb{C})\subset H_k\};\\
    I_2 &:= \mathscr{I};\\
    I_3&:=\{1,\dots, 2n+1\}\setminus (I_1\cup I_2).
  \end{align*} 
   If $I_3=\emptyset$, then
  the conclusion follows from Lemma~\ref{L: rational}. 
  And if $I_2=\emptyset$, then it follows from Lemma~\ref{L: Nochka Picards type}\,---\,owing to the discussions in
  Section~\ref{S:notions},
  $f$ intersects $H_k$ with the multiplicity $+\infty$ if $f(\mathbb{C})\subset H_k$ 
  or $f(\mathbb{C})\cap H_k=\emptyset$, and
 with the understanding that $1/(+\infty)=0$\,---\,that $f$ is a constant map.
  If $I_2\cup I_3=\emptyset$, then again $f$ is constant because $H_1,\dots, H_{2n+1}$
  are in general position.
  Therefore, we assume that $I_2\neq\emptyset\neq I_3$ and
 consider the following:
 \smallskip
 
 {\bf Case 1.} $I_1=\emptyset$, i.e., $f(\mathbb{C})\not\subset H_k$ for each $k\in\{1,\dots, 2n+1\}$.  
 It follows from the Second Main Theorem for $s$-nondegenerate  holomorphic curve 
 that the following  inequality holds for 
 all $r$ outside a set of finite Lebesgue measure:
 \begin{align*}
 sT_f(r)&\leq \sum_{k=1}^{2n+1}N_f^{[s]}(r, H_k)+O(\log (rT_f(r)))\\
           &\leq \sum_{k\in I_2} M_k \log r + \sum_{k\in I_3}\frac{s}{m_k} T_f(r) + O(\log (rT_f(r))),
 \end{align*}
 where $M_k$ denotes the cardinality of the set of zeros of $\widetilde{f}^*H_k$
 for $k\in I_2$ . 
 Since $s>0$ and $\sum_{k\in I_3}(1/m_k)<1$, 
 we have 
 \begin{equation*}
 \lim_{r\rightarrow \infty} \frac{T_f(r)}{\log r}<\infty.
 \end{equation*}
 Hence, by Lemma~\ref{L: log growth}, $f$ is rational.
 \smallskip
 
 {\bf Case 2.}  $I_1\neq\emptyset$. Put $$X_{I_1}:=\bigcap_{k\in I_1}H_k.$$ 
 Clearly, $f(\mathbb{C})\subset X_{I_1}$, and we can identify $X_{I_1}$ with a projective space
 of dimension $(n-t)$, where $t=|I_1|$. Again,
 for $k\not\in I_1$, the restrictions $H_k\cap X_{I_1}=:H_k^*$ are in general positions in $X_{I_1}=\mathbb{P}^{n-t}$.
 At this stage, we appeal to the Second Main Theorem which
 yields the following inequality
 \begin{align*}
 (s+t)T_f(r)&=(2n+1-t-2(n-t)+s-1)T_f(r)\\
 &\leq\sum_{k\not\in I_1}N^{[s]}_f(r, H_k^*))+ O(\log(rT_f(r))).
 \end{align*} 
 Combining this inequality with conditions ($b (i)$), ($b (ii)$) and the First Main Theorem for holomorphic curves,
 we get the following inequality
 \begin{align*}
  \Big(s+t-\sum_{k\in I_3}(s/m_k)\Big)T_f(r)\leq  O(\log r) + O(\log(T_f(r))).
\end{align*}
 Since $\big(s+t-\sum_{k\in I_3}(s/m_k)\big)>0$, the above inequality, together with Lemma~\ref{L: log growth}, implies that
 $f$ is rational. 
 \end{proof}

We will use the following easily provable lemma\,---\,we will not discuss the proof of this lemma\,---\, 
in order to prove Theorem~\ref{T: Main Xu shared hyperplane}. 
\begin{lemma}\label{L: Polynomial}
Let $p$ be a non-constant polynomial in $\mathbb{C}$, and $n\in\mathbb{Z}_+$. 
If $p$ has a non-zero root of multiplicity $n$ then $p$ consists of at least $n+1$ terms
in its expression.
\end{lemma}

\section{The proof of Theorem~\ref{T: Main Xu shared hyperplane}}\label{S: proof Main}

\begin{proof}[Proof of Theorem~\ref{T: Main Xu shared hyperplane}]
 We begin the proof by setting
 \begin{align*}
  \mathcal{H}_a:=\{H_1,\dots,H_{n+1}\} {\text{ and }}
  \mathcal{H}_b:= \{H_{n+2},\dots, H_{2n+1}\}.
 \end{align*}
 We shall prove that $\mathcal{F}$ is normal at each point $z\in D$. Fix an arbitrary point $z_0\in D$ and 
 a holomorphic curve $f\in\mathcal{F}$. The collection $\mathcal{H}_a$ is partitioned into two 
 subsets
 \[
   \mathcal{H}^1(f) := \{H\in\mathcal{H}_a\,:\, f(z_0)\in H\}
   \quad\text{and} \quad
   \mathcal{H}^2(f) := \{H\in\mathcal{H}_a\,:\, f(z_0)\notin H\}.
 \]
 There exists a neighborhood $U_{z_0}\subset D$ such that
 \begin{itemize}
  \item if $H\in \mathcal{H}^2(f)$ then $f(U_{z_0})\cap H =\emptyset$; and
  \item if $H\in \mathcal{H}^1(f)$ and $f(D)\not\subseteq H$ then $f(U_{z_0})\cap H = \{f(z_0)\}$.
 \end{itemize} 
 This allows us to divide $\mathcal{H}_a$ into three disjoint subsets 
 $\mathcal{H}_1, \mathcal{H}_2, \mathcal{H}_3$ as follows:
 \begin{align*}
  \mathcal{H}_1 & := \{H\in\mathcal{H}_a\,:\, f(U_{z_0})\subseteq H\};\\
  \mathcal{H}_2 & := \{H\in\mathcal{H}_a\,:\,  f(U_{z_0})\cap H =\emptyset \};\\
  \mathcal{H}_3 & :=\{H\in\mathcal{H}_a\,:\, f(U_{z_0})\not\subseteq H \ \text{ and } 
   f(U_{z_0})\cap H = f(z_0) \}.
 \end{align*} 
 By the the condition~$(a)$
 of Theorem~\ref{T: Main Xu shared hyperplane}, we see that the above choice of $U_{z_0}$\,---\,which we
 had initially made for the chosen $f$\,---\,is independent of the choice of $f\in\mathcal{F}$. Hence, the sets
 $\mathcal{H}_i$, $i=1, 2, 3$, are independent of $f\in\mathcal{F}$ as well.
 \smallskip

 We now fix an arbitrary sequence $\{f_j\}$. We shall  
 prove that the sequence $\big\{\!\left.f_j\right|_{U_{z_0}}\!\big\}$
 has a subsequence that converges uniformly on compact subsets of
 $U_{z_0}$. Let us assume that this is not true, and aim 
 for a contradiction. Then, by Lemma~\ref{L: AK},
 there exist a point $ z'_0\in U_{z_0}$ and
 \begin{enumerate}
   \item[(i)] a subsequence of $\big\{\!\left.f_j\right|_{U_{z_0}}\!\big\}$, which we may label 
    without causing confusion as $\big\{\!\left.f_j\right|_{U_{z_0}}\!\big\}$;
   \item[(ii)] a sequence $\{z_j\}\subset {U_{z_0}}$ such that $z_j\rightarrow z'_0$;
   \item[(iii)] a sequence  $\{r_j\}\subset \mathbb{R}$ with $r_j>0$ such that $r_j\rightarrow0$
 \end{enumerate}
 such that\,---\,defining the maps $h_j: \zeta\mapsto f_j(z_j+r_j\zeta)$ on suitable
 neighborhoods of $0\in \mathbb{C}$\,---\,$\{h_j\}$
 converges uniformly on compact subsets of $\mathbb{C}$ to a {\it non-constant}
 holomorphic mapping $h:\mathbb{C}\rightarrow\mathbb{P}^n$. Then, there exist reduced representations
 \begin{equation}\label{eq: adm repr}
  \widetilde{f}_j = \big(f_{j,\,0}, f_{j,\,1},\ldots, f_{j,\,n}\big); \   
  \widetilde{h}_j = \big(h_{j,\,0}, h_{j,\,1},\ldots, h_{j,\,n}\big); \text{ and }
  \widetilde{h} = \big(h_0, h_1,\ldots, h_n\big)
 \end{equation}
 of $f_j$, $h_j$, and $h$ respectively (that are defined globally on their respective domains of definition)
 such that 
 $\widetilde{h}_j(\zeta):=\widetilde{f}_j(z_j+r_j\zeta)\rightarrow  \widetilde{h}(\zeta)$ 
 uniformly on compact subsets of $\mathbb{C}$\,---\,see \cite[\S5--6]{Dufresnoy}. This implies
 that $\{ \widetilde{h}_j ^*H_k\}$
 converges uniformly on compact subsets of $\mathbb{C}$ to $\widetilde{h}^*H_k$, $1\leq k\leq 2n+1$.
 Clearly,  $\widetilde{h}^*H_k$ is holomorphic on the entire complex plane $\mathbb{C}$ for all $k= 1,\dots, 2n+1$.
 Recall that we defined $\widetilde{f}^*H$, where $H$ is any hyperplane, in Section~\ref{S:notions}. 
 \smallskip
 
 We, now, aim for the following {\bf claim~1:}  {\it For each hyperplane  $H\in\mathcal{H}_a$, 
  either $h(\mathbb{C})\subset H$ or $\mathsf{supp}\,\nu(h, H)$ consists of at most one 
  point in $\mathbb{C}$.}
\smallskip

 Clearly, by definitions of $\mathcal{H}_1, \mathcal{H}_2$ and Hurwitz's Theorem, 
 for each $H\in \mathcal{H}_1\cup\mathcal{H}_2$, one of the following holds:
 \begin{equation}\label{eq: hc H MT}
  h(\mathbb{C})\subset H \quad \text{or} \quad h(\mathbb{C})\cap H=\emptyset.
 \end{equation}
 
 If $\mathcal{H}_3=\emptyset$ then our claim is established. We may thus assume that 
 $\mathcal{H}_3\neq \emptyset$ and
 fix an arbitrary hyperplane $H\in\mathcal{H}_3$. We may, without loss of generality, 
 assume that $\widetilde{h}^*H\not\equiv c$, where $c\in \mathbb{C}$.
 We shall prove that $\widetilde{h}^*H$ has a unique zero. 
 Suppose on the contrary that $\widetilde{h}^*H$ has at least two distinct zeros in $\mathbb{C}$. 
 Let $\zeta_1$ and $\zeta_2$ be any two distinct zeros of $\widetilde{h}^*H$ in $\mathbb{C}$, 
 then there exist two disjoint neighborhoods 
 $U_{\zeta_1}\ni\zeta_1$ and $U_{\zeta_2}\ni\zeta_2$ in $\mathbb{C}$ such that $\widetilde{h}^*H$ vanishes
 only at $\zeta_1$ and $\zeta_2$ in $U_{\zeta_1}\cup U_{\zeta_2}$. By Hurwitz's Theorem, there exist 
 sequences $\{\zeta^1_j\}$ and $\{\zeta_j^2\}$ converging to $\zeta_1$ and $\zeta_2$ respectively such that 
 for $j$ sufficiently large we have 
 \begin{equation*}
   \widetilde{h}^*_jH(\zeta_j^1)=\widetilde{f}_j^*H(z_j+r_j\zeta_j^1)=0 
   \text{ and } \widetilde{h}^*_jH(\zeta^2_j)=\widetilde{f}_j^*H(z_j+r_j\zeta_j^2)=0.
 \end{equation*}
 By Condition~$(a)$ of Theorem~\ref{T: Main Xu shared hyperplane}, for each $f_i$ in $\mathcal{F}$ we have 
  \[
   \widetilde{f}_i^*H(z_j+r_j\zeta_j^1)=0 \text{ and } \widetilde{f}_i^*H(z_j+r_j\zeta_j^2)=0.
 \]
 Now we fix $i$ and let $j\rightarrow\infty$, and notice that $z_j+r_j\zeta_j^1\rightarrow z'_0$
 and $z_j+r_j\zeta_j^2\rightarrow z'_0$. Thus we get 
 $\widetilde{f}_i^*H(z'_0)=0$, whence we deduce, from the definition of $\mathcal{H}_3$, that
 (whenever $\mathcal{H}_3\neq \emptyset$)
 $z'_0=z_0$. Since the zeros of non-constant univariate holomorphic functions are isolated, we have $z_j+r_j\zeta^1_j=z_0$
 and $z_j+r_j\zeta^2_j=z_0$ for sufficiently large $j$. Hence $\zeta^1_j=(z_0-z_j)/ r_j=\zeta^2_j$, which is not
 possible as $\zeta^1_j\to \zeta^1$ and  $\zeta^2_j\to \zeta^2$, but $\zeta^1\neq \zeta^2$.
 Therefore $\widetilde{h}^*H$ has a unique zero,  
 and hence $\mathsf{supp}\,\nu(h, H)$ consists of only one point. This, together with \eqref{eq: hc H MT}
 establishes {\bf claim 1}. At this point, we make a further observation that will be important.  
 Set
 \[
  \xi_0:=\lim_{j\rightarrow\infty} \frac{z_0-z_j}{ r_j}.
 \]
 From the above discussion, we notice that $\xi_0$ is the {\bf unique} zero
 of $\widetilde{h}^*H$ in $\mathbb{C}$, where $H$ is as chosen at the beginning of this paragraph. Now,
 as the parameters $z_0$, $z_j$ and $r_j$ are independent of $H\in\mathcal{H}_3$, we get
 \begin{itemize}
  \item[$(*)$] {\it If $\mathcal{H}_3\neq \emptyset$ then there exists a
  $\xi_0\in \mathbb{C}$ that is the unique zero of $\widetilde{h}^*H$
  for each $H\in \mathcal{H}_3$.}
 \end{itemize}
 \smallskip
 
 We, next, aim for the following {\bf claim 2:} 
{\it For each hyperplane $H$ in $\mathcal{H}_b$, the map $h$ intersects $H$
 with multiplicity at least $n+1$.}
\smallskip 

 Fix an arbitrary hyperplane $H\in\mathcal{H}_b$. If $f_j(U_{z_0})\cap H=\emptyset$
 for sufficiently large $j$, then by Hurwitz's Theorem, either $\widetilde{h}^*H\equiv 0$ or 
 $\widetilde{h}^*H$ is non-vanishing on $\mathbb{C}$. Hence the map $h$ intersects $H$ with the multiplicity $\infty$
 and our claim is established. Thus, we may assume 
 that $\widetilde{h}^*H\not\equiv 0$ and $\widetilde{h}^*H$ has zeros in $\mathbb{C}$.
 Suppose $\zeta_0\in\mathbb{C}$ is such that $\widetilde{h}^*H(\zeta_0)=0$. 
 As $h_0,\dots, h_n$ cannot simultaneously vanish,
 there exists an index $l_0\in\{0,\dots, n\}$ and an open disc $B(\zeta_0, r)\subset\mathbb{C}$,
 centred at $\zeta_0$ with radius $r$, such that\,---\,fixing homogeneous 
 coordinates on $\mathbb{P}^n$\,---\,the image $h(B(\zeta_0, r))$ is a subset of
 $V_{l_0}:=\{[w_0:\dots:w_n]\in\mathbb{P}^n: w_{l_0}\neq 0\}$. Thus, in the reduced representation 
 $(h_0,\dots,h_n)$ of $h$ (given in~\eqref{eq: adm repr})  we have $h_{l_0}(\zeta)\neq 0$ on $B(\zeta_0, r)$. 
  Hurwitz's Theorem implies that $h_{j,\,l_0}(\zeta):=f_{j,\,l_0}(z_j+r_j\zeta)\neq 0$ on $B(\zeta_0, r)$ when $j$ is large enough.
 Now, we consider the following reduced representaions  
 $\widetilde{\nh}_j=(h_{j;\,0}/h_{j,\,l_0},\dots, h_{j,\,n}/h_{j,\,l_0})$ of $h_j$
 and  $\widetilde{\nh}=(h_{0}/h_{l_0},\dots, h_{n}/h_{l_0})$ of $h$ on $B(\zeta_0, r)$. 
 Thus, for every $j$ large enough, we get connected open subsets $U_{j}\ni z'_0$ of $D$
 and reduced representations $\widetilde{\nf}_j$ of 
 $f_j$ defined on $U_j$ such that 
 the $l_0$-th coordinate of $\widetilde{\nf}_j$ is identically $1$ on $U_{j}$, and 
 $\widetilde{\nh}_j(\zeta)=\widetilde{\nf}_j(z_j+r_j\zeta)$ on $B(\zeta_0, r)$.
It is clear that 
 $\widetilde{\nh}_j\rightarrow\widetilde{\nh}$ converges uniformly on compact subsets of $B(\zeta_0, r)$.
 Hence $\widetilde{\nh}_j^*H\rightarrow \widetilde{\nh}^*H$
 converges uniformly on compact subsets of $B(\zeta_0, r)$, which further implies that 
 \begin{equation}\label{eq: MT k-th derivative}
  (\widetilde{\nh}_j^*H)^{(q)}\rightarrow (\widetilde{\nh}^*H)^{(q)}
 \end{equation}
 converges uniformly on compact subsets of $B(\zeta_0, r)$ for all $q\in\mathbb{Z}_+$.
 \smallskip
   
 At this stage, we notice that  
 $z_j+r_j\zeta$ belong to a compact subset $K_0\ni z'_0$ of $D$ for all $\zeta\in B(\zeta_0, r)$ 
 if $j$ is sufficiently large\,---\,this holds
 true because $z_j\rightarrow z'_0$, $r_j\rightarrow 0$, and $B(\zeta_0, r)\Subset\mathbb{C}$. 
 From Condition~$(b)$ of Theorem~\ref{T: Main Xu shared hyperplane},
 there exists a positive integer $M$ such that
 \begin{equation*}
   \big|(\widetilde{\nf}_j^*H)^{(q)}(z)\big|_{_{1\leq q\leq n}}\leq M, \quad z\in f_j^{-1}(H)\cap {K_0},
 \end{equation*}
 hold for $H$ (which we had fixed in the previous paragraph) and $j\in\mathbb{Z}_+$. 
 Recall that $\widetilde{h}^*H(\zeta_0)=0$. By 
 Hurwitz's Theorem, we get a sequence $\zeta_{j}\rightarrow \zeta_0$ such that  
 $\widetilde{h}_j^*H(\zeta_j)=\widetilde{f}_j^*H(z_j+r_j\zeta_j)=0$ for $j$ sufficiently large. Which then implies that for  
 sufficiently large $j$ we have $z_j+r_j\zeta_j\in f_j^{-1}(H)$. Hence for $j$ large enough we get that
 \begin{equation*}
   |(\widetilde{\nf}_{j}^*H)^{(q)}(z_j+r_j\zeta_j)|\leq M
 \end{equation*}
 for all $q=1,\dots, n$. This further implies that for all $q=1,\dots, n$, we have
 \begin{equation}\label{eq: MT bound on kth derivative of hj}
   \left|(\widetilde{\nh}_{j}^*H(\zeta))^{(q)}\Big|_{\zeta=\zeta_j}\right|=
   |r_j^q(\widetilde{\nf}_j^*H)^{(q)}(z_j+r_j\zeta_j)|\leq r_j^qM.
 \end{equation}
 Combining \eqref{eq: MT k-th derivative} and \eqref{eq: MT bound on kth derivative of hj}, 
 together, we have
 \begin{equation*}
   (\widetilde{\nh}^*H)^{(q)}(\zeta_0)
   = \lim_{j\rightarrow\infty} (\widetilde{\nh}_j^*H)^{(q)}(\zeta_j)
   =0
 \end{equation*}
 for all $q=1,\dots, n$. Hence $\zeta_0$ is a zero of $\widetilde{\nh}^*H$ of multiplicity at least $n+1$. 
 This confirms that $h$ intersects $H$ with multiplicity at least $n+1$. Since $H\in\mathcal{H}_b$
 was arbitrarily chosen, the latter conclusion holds true for each $H\in\mathcal{H}_b$, which establishes
 {\bf claim 2}.
 \smallskip
   
 At this stage, we appeal to Proposition~\ref{P: rational}\,---\,in view of  {\bf claim~1} and {\bf claim~2}\,---\,to 
 conclude that $h$ is {\bf rational}. 
  We now analyze what consequence this has for $\widetilde{h}^*H$ when $H\in \mathcal{H}_a$.
  From $(*)$, \eqref{eq: hc H MT}, and the fact that each $H\in \mathcal{H}_a$ is a coordinate
 hyperplane, we deduce that there exist $m_l\in\mathbb{Z}_+$, 
 $b_l\in\mathbb{C}\setminus\{0\}$, and $c_l\in\mathbb{C}$  such that
 either 
 \begin{align*}
  \widetilde{h}^*H_{l+1} \equiv b_l (z-\xi_0)^{m_l}, \text{ or } 
  \widetilde{h}^*H_{l+1} \equiv c_l \text{ for } l=0,\dots, n
 \end{align*}
 (the first case applies precisely to those hyperplanes in $\mathcal{H}_3$ in case $\mathcal{H}_3\neq \emptyset$).
 We now divide 
 the set $\{0, 1, \dots, n\}$ into two disjoint subsets such that 
 \begin{align*}
  I_1&:=\{l\,:\, H_{l+1}\in\mathcal{H}_a {\text{ and }} \widetilde{h}^*H_{l+1}\equiv b_l (z-\xi_0)^{m_l}\};\\
  I_2&:= \{l\,:\, H_{l+1}\in\mathcal{H}_a {\text{ and }} \widetilde{h}^*H_{l+1}\equiv c_l\}.
 \end{align*}
       
 Since $h$ is non-constant holomorphic map, and the
 hyperplanes in $\mathcal{H}_a\cup\mathcal{H}_b$ are in general position, 
 we may assume that $I_1\neq \emptyset\neq I_2$, 
 and that there exists at least one index $l\in I_2$ with $c_l\neq 0$. 
 This, therefore, implies that in the representation 
 $(h_0,\dots, h_n)=:\widetilde{h}$ of $h$ (given in~\eqref{eq: adm repr}), each $h_l$,  $l=0,\dots, n$, is 
 either the constant $c_l$ or the monomial $b_l(z-\xi_0)^{m_l}$. 
 Since 
 $\mathcal{H}_a$ consists of all the coordinate hyperplanes, this implies that $\mathcal{H}_b$ consists of hyperplanes 
 of the following kind: $H_k:=\big\{[w_0 : \dots : w_n]\in\mathbb{P}^n\,|\,\sum_{l=0}^{n}a_l^kw_l=0\big\}$ 
 where  $a_l^k\neq 0$ for all 
 $l=0,\dots, n$, and all $k=n+2, \dots, 2n+1$. Thus, we have  
 $$\widetilde{h}^*H_k=\sum_{l\in I_1} B_l^k(z-\xi_0)^{m_l}+C^k, \quad B_l^k, C^k\in \mathbb{C};$$
 for $k=n+2, \dots, 2n+1$.   
 We now assume, without loss of generality, that $\xi_0=0$. 
 Thus, $\widetilde{h}^*H_k=\sum_{l\in I_1} B_l^kz^{m_l}+C^k$, $k=n+2, \dots, 2n+1$. 

 \smallskip
    
 Suppose $\widetilde{h}^*H_k$ are identically constants for all $k=n+2, \dots, 2n+1$. 
 Owing to the fact that $I_2\neq\emptyset$,   
 there exist a subset $\mathcal{H}''\subset\mathcal{H}_a\cup\mathcal{H}_b$ 
 consists of at least $n+1$ hyperplanes such that 
 \[
  \widetilde{h}^*H_k= c_k, \quad H_k\in \mathcal{H}''.
 \]
 The above system of equations confirms that $h_0, \dots, h_n$ must be constants\,---\,as 
 given hyperplanes are in general position, and the cardinality of $\mathcal{H}''$ is at least $n+1$\,---\,which is a contradiction.
 Next, suppose that there exists at least one index $k\in\{n+2, \dots, 2n+1\}$ such that 
 $\widetilde{h}^*H_k$ is a non-constant polynomial.
 Since hyperplanes are in general position, there must exists an index $k_0\in\{n+2,\dots, 2n+1\}$ such that 
 $\widetilde{h}^*H_{k_0}=\sum_{l\in I_1} B_l^{k_0}z^{m_l}+C^{k_0}$
 is a non-constant polynomial such that $B_l^{k_0}, C^{k_0}\in\mathbb{C}\setminus\{0\}$.
 Clearly, $\widetilde{h}^*H_{k_0}$ consists of at most $n+1$ terms in 
 its expression and $\widetilde{h}^*H_{k_0}$ has a non-zero root because $C^{k_0}\neq 0$. 
 We now invoke Lemma~\ref{L: Polynomial} to conclude that 
 $\widetilde{h}^*H_{k_0}$ has  zeros of multiplicity at most 
 $n$, which is, in view of {\bf claim~2}, again a contradiction. Thus, the sequence $\big\{\!\left.f_j\right|_{U_{z_0}}\!\big\}$
 has a subsequence that converges uniformly on compact subsets of
 $U_{z_0}$. 
 \smallskip
 
 Recall now that $z_0\in D$ was chosen arbitrarily. Since we can cover $D$ by a 
 countable collection of open sets $U_{z_0}$
 \begin{itemize}
   \item where $z_0$ varies through some countable dense subset of $D$; and
   \item $U_{z_0}$ is such that it has the properties given by the two bullet-points at the beginning of this proof;
 \end{itemize}
 a standard diagonal argument gives us a subsequence of $\{f_j\}$ that converges uniformly on compact subsets of
 $D$. This completes the proof.
\end{proof}


\section{The proof of Theorem~\ref{T: vol of divisor}}\label{S: proof vol of divisor}

\begin{proof}[Proof of Theorem\,\ref{T: vol of divisor}]

 We begin the proof by dividing the set $\{H_1,\dots, H_{2n+1}\}$ of hyperplanes into three disjoint 
 subsets
 \[
 \mathcal{H}_{a}:=\{H_1, \dots, H_{n+1}\}; \quad \mathcal{H}_{b}:=\{H_{n+2}, \dots, H_{t}\}; \quad \text{and }
 \mathcal{H}_{c}:=\{H_{t+1}, \dots, H_{2n+1}\}; \]
where $t$ is a positive integer such that $n+2\leq t<2n+1$, and $n\geq 2$. Existence of $t$ is ensured by the 
assumption stated just after condition $(a)$ of the Theorem~\ref{T: vol of divisor}.
 We shall prove that $\mathcal{F}$ is normal at each point $z\in D$. Fix an arbitrary point $z_0\in D$ and 
 a holomorphic curve $f\in\mathcal{F}$.
 Now we repeat the argument, {\it mutatis mutandis}, in the first paragraph of the 
 proof of Theorem~\ref{T: Main Xu shared hyperplane} to get a neighborhood $U_{z_0}\ni z_0$ in $D$ so that 
 $\mathcal{H}_a\cup\mathcal{H}_b=\{H_1,\dots, H_t\}$ can be divided into the following 
 three disjoint (not necessarily non-empty) subsets 
 \begin{align*}
  \mathcal{H}_1 & := \{H\in \mathcal{H}_a\cup\mathcal{H}_b\,:\,  f(U_{z_0})\cap H =\emptyset \};\\
  \mathcal{H}_2 & := \{H\in \mathcal{H}_b\,:\, f(U_{z_0})\subseteq H\};\\
  \mathcal{H}_3 & :=\{H\in \mathcal{H}_b\,:\, f(U_{z_0})\not\subseteq H \ \text{ and } 
   f(U_{z_0})\cap H =\{f(z_0)\} \}.
 \end{align*} 
 Clearly, $\mathcal{H}_a\subset \mathcal{H}_1$, whence $\mathcal{H}_1\neq\emptyset$.
 By the conditions~$(a)$ and $(b)$ of Theorem~\ref{T: vol of divisor}, 
 we deduce that the open set $U_{z_0}$ and the sets 
 $\mathcal{H}_i$, $i=1, 2, 3$, are independent of $f\in\mathcal{F}$.
 \smallskip

 We now fix an arbitrary sequence $\{f_j\}\subset\mathcal{F}$.
 We shall  prove that the sequence 
$\big\{\!\left.f_j\right|_{U_{z_0}}\!\big\}$ has a subsequence that converges compactly on $U_{z_0}$.
 Let us assume that the latter is not true, and aim for a contradiction. Then, by Lemma~\ref{L: AK},
 there exist a point  $z'_0\in U_{z_0}$ and
 \begin{enumerate}
  \item[(i)] a subsequence of $\big\{\!\left.f_j\right|_{U_{z_0}}\!\big\}$, which we may label 
  without causing confusion as $\big\{\!\left.f_j\right|_{U_{z_0}}\!\big\}$;
  \item[(ii)] a sequence $\{z_j\}\subset U_{z_0}$ such that $z_j\rightarrow z'_0$;
  \item[(iii)] a sequence  $\{r_j\}\subset \mathbb{R}$ with $r_j>0$ such that $r_j\rightarrow0$
 \end{enumerate}
 such that\,---\,defining the maps $h_j: \zeta\mapsto f_j(z_j+r_j\zeta)$ on suitable
 neighborhoods of $0\in \mathbb{C}$\,---\,$\{h_j\}$
 converges uniformly on compact subsets of $\mathbb{C}$ to a {\it non-constant}
 holomorphic mapping $h:\mathbb{C}\rightarrow\mathbb{P}^n$. 
 Then, there exist reduced representations
 \begin{equation*}\label{eq: adm repr2}
  \widetilde{f}_j = \big(f_{j,\,0}, f_{j,\,1},\ldots, f_{j,\,n}\big); \   
  \widetilde{h}_j = \big(h_{j,\,0}, h_{j,\,1},\ldots, h_{j,\,n}\big); \text{ and }
  \widetilde{h} = \big(h_0, h_1,\ldots, h_n\big)
 \end{equation*}
 of $f_j$, $h_j$, and $h$ respectively (that are defined globally on their respective domains of definition) such that 
 $\widetilde{h}_j(\zeta):=\widetilde{f}_j(z_j+r_j\zeta)\rightarrow  \widetilde{h}(\zeta)$ 
 uniformly on compact subsets of $\mathbb{C}$. 
 Clearly, by the definitions of $\mathcal{H}_1$ and $\mathcal{H}_2$, and by Hurwitz's Theorem, 
 for each $H\in \mathcal{H}_1\cup\mathcal{H}_2$, one of the following holds:
 \begin{equation}\label{eq hc H}
  h(\mathbb{C})\subset H \text{ or } h(\mathbb{C})\cap H=\emptyset.
 \end{equation}
 We now aim for the following {\bf claim:}
{\it For each hyperplane $H$ in $\mathcal{H}_3\cup\mathcal{H}_c$, either
 $h(\mathbb{C})\subset H$ or $\widetilde{h}^*H$ has at most finitely many zeros in $\mathbb{C}$.}
\smallskip

 If $\mathcal{H}_3=\emptyset$, then our claim is established 
 for the case where hyperplanes are in $\mathcal{H}_3$. Thus, we may 
 assume that $\mathcal{H}_3\neq \emptyset$ and fix an arbitrary hyperplane $H\in\mathcal{H}_3$.
 Now we repeat verbatim the argument that we used to prove {claim~1} in the proof of 
 Theorem~\ref{T: Main Xu shared hyperplane} to establish our {\bf claim} in this case.
 We now aim for establishing the {\bf claim} for the case where hyperplanes are in $\mathcal{H}_c$.
 Fix an arbitrary hyperplane $H\in\mathcal{H}_c$. 
 If $f_j(U_{z_0})\cap H=\emptyset$ for sufficiently large $j$, then by Hurwitz's Theorem,
 either $h(\mathbb{C})\subset H$
 or $\widetilde{h}^*H$ has no zero in $\mathbb{C}$, whence our claim is established.  Thus, we may assume that 
 $\widetilde{h}^*H\not\equiv 0$ on $\mathbb{C}$. 
 At this stage, we notice that  
 $z_j+r_j\zeta$ belong to a compact subset $K_0\ni z'_0$ of $D$ for all $\zeta$ contained in a 
 compact subset of $\mathbb{C}$ if $j$ is large enough\,---\,this holds 
 true because $\zeta$ lies in a compact subset of $\mathbb{C}$, $z_j\rightarrow z'_0$, and $r_j\rightarrow 0$.
 From Condition~(c) of Theorem~\ref{T: vol of divisor}, 
 there exists a finite integer $M\in\mathbb{Z}_+$ such that the set of zeros, counting with multiplicities, of
 $\widetilde{f}_j^*H$ in $K_{0}$ is of cardinality less than $M$ for all $j$. This implies that
 the cardinality of the set of zeros, counting with multiplicities, of
 $\widetilde{h}_j^*H$\,---\,in every compact subsets of $\mathbb{C}$, 
 and hence in $\mathbb{C}$\,---\,is less than the positive integer $M$ for all $j$.
 By the virtue of uniform convergence and the Argument Principle (or, by using Hurwitz's Theorem), we 
 conclude that $\widetilde{h}^*H$ has finitely many (less than $M$ in numbers) zeros 
 in $\mathbb{C}$, which establishes our {\bf claim}.
 \smallskip
 
 From Lemma~\ref{L: rational}\,---\,in view of \eqref{eq hc H} and 
 the {\bf claim}\,---\,we conclude that $h$ is rational.
 This implies that 
  $\widetilde{h}^*H_k$ is a polynomial for each $k=1,\dots,2n+1$.
 Therefore, we deduce, from \eqref{eq hc H}, that there exist
 constants $c_k\in\mathbb{C}$ such that  $\widetilde{h}^*H_k\equiv c_k$ 
 for $H_k\in  \mathcal{H}_1\cup \mathcal{H}_2$.
 Recall that the hyperplanes are in general position and the cardinality of 
 $\mathcal{H}_1\cup \mathcal{H}_2$ is at least $n+1$. 
 Thus, for any subset $\mathcal{H}''\subset\mathcal{H}_1\cup \mathcal{H}_2$ of cardinality $n+1$
 the  following system of equations:
 \[
  \widetilde{h}^*H_k=c_k, \quad H_k\in \mathcal{H}''
 \] 
 confirms that $h_0, \dots, h_n$ must be constant. Hence, $h$ is a constant map. Which is a contradiction.
 Thus the sequence $\big\{\!\left.f_j\right|_{U_{z_0}}\!\big\}$ has a subsequence that converges uniformly on compact 
 subsets of $U_{z_0}$. We are now in a position to repeat verbatim the argument 
 in the final paragraph of the proof of Theorem~\ref{T: Main Xu shared hyperplane}
 to conclude that  $\{f_j\}$ has a subsequence that converges uniformly on compact subsets of $D$. This completes 
 the proof.
 \end{proof}

\section*{Acknowledgements}\vspace{-1mm}
I wish to express my sincere gratitude to Gautam Bharali for helpful 
discussions in the course of this work.
This work is supported by the Dr.\,D.\,S. Kothari Postdoctoral Fellowship of UGC, India
(Grant no.~No.F.4-2/2006 (BSR)/MA/19-20/0022) and a UGC CAS-II grant (No.~F.510/25/CAS-II/2018(SAP-I)).

\end{document}